\documentclass[12pt]{amsart}

\usepackage{graphicx,amsmath,amscd,amsfonts,amsthm,amssymb,verbatim,stmaryrd,fullpage,enumerate}
\newtheorem{theorem}{Theorem}
\newtheorem{lemma}[theorem]{Lemma}
\newtheorem{prop}[theorem]{Proposition}

\theoremstyle{remark}

\newcommand{\R}{\mathbb R}
\newcommand{\C}{\mathbb C}

\newcommand{\Z}{\mathbb Z}
\newcommand{\Q}{\mathbb Q}
\newcommand{\F}{\mathbb F}
\newcommand{\A}{\mathbb A}
\newcommand{\bH}{\mathbb H}

\newcommand{\cO}{\mathcal O}

\newcommand{\cI}{\mathcal I}

\newcommand{\be}{\begin{equation}}
\newcommand{\ee}{\end{equation}}
\newcommand{\bes}{\begin{equation*}}
\newcommand{\ees}{\end{equation*}}

\newcommand{\ba}{\begin{eqnarray}}
\newcommand{\ea}{\end{eqnarray}}
\newcommand{\bas}{\begin{eqnarray*}}
\newcommand{\eas}{\end{eqnarray*}}

\title{Local bounds for $L^p$ norms of Maass forms in the level aspect}
\author{Simon Marshall}

\begin{document}

\begin{abstract}

We apply techniques from harmonic analysis to study the $L^p$ norms of Maass forms of varying level on a quaternion division algebra.  Our first result gives a candidate for the local bound for the sup norm in terms of the level, which is new when the level is not squarefree.  The second result is a bound for $L^p$ norms in the level aspect that is analogous to Sogge's theorem on $L^p$ norms of Laplace eigenfunctions.

\end{abstract}

\maketitle

\section{Introduction}

Let $\phi$ be a cuspidal newform of level $\Gamma_0(N)$ on $GL_2/\Q$ or a quaternion division algebra over $\Q$, which we shall assume is $L^2$-normalised with respect to the measure that gives $\Gamma_0(N) \backslash \bH^2$ mass 1.  There has recently been interest in bounding the sup norm $\| \phi \|_\infty$ in terms of $N$ and the infinite component of $\phi$, see \cite{BH, HT1, HT2, Sa, T3, T1, T2}.  The `trivial' bound in the level aspect (with the infinite component remaining bounded) is generally considered to be $\| \phi \|_\infty \ll_\epsilon N^{1/2 + \epsilon}$, provided $N$ is squarefree; see \cite{AU} or any of the previously cited papers.  Our first result is a candidate for the generalisation of this to arbitrary $N$.

\begin{theorem}
\label{level}

Let $D/\Q$ be a quaternion division algebra that is split at infinity.  Let $\phi$ be an $L^2$-normalised newform of level $K_0(N)$ on $PGL_1(D)$.  Assume that $\phi$ is spherical at infinity with spectral parameter $t$.  Let $N_0 \ge 1$ be the smallest number with $N | N_0^2$.  We have

\bes
\| \phi \|_\infty \ll (1 + |t|)^{1/2} N_0^{1/2} \prod_{p | N} (1 + 1/p)^{1/2}.
\ees

\end{theorem}

Notation is standard, and specified below.  When $t$ is bounded, the Theorem gives a bound of $N^{1/2 + \epsilon}$ for $N$ squarefree, but roughly $N^{1/4 + \epsilon}$ for powerful $N$.  While the Theorem is restricted to compact quotients, $\mathsection$\ref{noncomp} gives a weaker result in the case of $PGL_2/\Q$.

Our second result is the analogue in the level aspect of a classical theorem of Sogge \cite{So}, which we now recall.  Let $M$ be a compact Riemannian surface with Laplacian $\Delta$, and let $\psi$ be a function on $M$ satisfying $(\Delta + \lambda^2)\psi = 0$ and $\| \psi \|_2 = 1$.  Define $\delta : [2,\infty] \rightarrow \R$ by

\begin{equation}
\label{delta0}
\delta(p) = \bigg\{ \begin{array}{ll} \tfrac{1}{2} - \tfrac{2}{p} , & 0 \le \tfrac{1}{p} \le \tfrac{1}{6}, \\
\tfrac{1}{4} - \tfrac{1}{2p}, & \tfrac{1}{6} \le \tfrac{1}{p} \le \tfrac{1}{2}.
\end{array}
\end{equation}
Sogge's theorem states that

\be
\label{eval}
\| \psi \|_p \ll \lambda^{\delta(p)} \quad \text{for} \quad 2 \le p \le \infty.
\ee
In particular, this is stronger than the bound obtained by interpolating between bounds for the $L^2$ and $L^\infty$ norms.  Our next theorem demonstrates that something similar is possible in the level aspect.

\begin{theorem}
\label{Sogge}

Let $D/\Q$ be a quaternion division algebra that is split at infinity.  Let $\phi$ be an $L^2$-normalised newform of level $K_0(q^2)$ on $PGL_1(D)$, where $q$ is a prime.  Assume that $\phi$ is principal series at $q$.  Assume that $\phi$ is spherical at infinity with spectral parameter $t$, and that $|t| \le A$ for some $A > 0$.  We have

\bes
\| \phi \|_p \ll_A q^{\delta(p)}.
\ees

\end{theorem}

It should be possible to give some extension of Theorem \ref{Sogge} to general $\phi$, although in some cases the method may not give any improvement over the bound given by interpolating between $L^2$ and $L^\infty$ norms.  In particular, this seems to occur when $\phi$ is special at $q$.  We have chosen to work in the simplest case where the method gives a non-trivial result.

We guess that Theorems \ref{level} and \ref{Sogge} are the correct local bounds for $L^p$ norms in the level aspect, in the same way that (\ref{eval}) is the local bound in the eigenvalue aspect.  The term `local bound' means the best bound that may be proved by only considering the behaviour of $\phi$ in one small open set at a time, without taking the global structure of the space into account.  Equivalently, this is the sharp bound for wave packets localised at scale 1.

We make this guess for two reasons.  The first is that the analogue of Theorem \ref{Sogge} on the `compact form' $PGL(2,\Z_q)$ of the arithmetic quotient being considered is sharp when $p \ge 6$, provided one takes a vector of the same type as $\phi'$ defined below.  This may be seen from equation (\ref{innerprod}) and Lemma \ref{matrix}.  The same should be true for Theorem \ref{level} when $N$ is a growing power of a fixed prime.  Secondly, we expect the bound of Theorem \ref{level} to have a natural expression as the square root of the Plancherel density around the representation of $\phi$. Because the proofs do not make use of the global structure of the arithmetic quotient, it should be possible to improve the exponents by using arithmetic amplification.

{\bf Acknowledgements:}  We would like to thank Valentin Blomer, Abhishek Saha, and Nicolas Templier for helpful discussions.

\section{Notation}

\subsection{Adelic groups}

Let $D/\Q$ be a quaternion division algebra that is split at infinity.  Let $S$ be the set containing 2 and all primes that ramify in $D$, and let $S_\infty = S \cup \{ \infty \}$.  Let $G = PGL_1(D)$.  If $v$ is a place of $\Q$, let $G_v = G(\Q_v)$.  Let $X = G(\Q) \backslash G(\A)$.  Let $\cO \subset D$ be a maximal order.  Let $K = \otimes_p K_p \subset G(\A_f)$ be a compact subgroup with the property that $K_p$ is open in $G_p$ for $p \in S$, and $K_p$ is isomorphic to the image of $\cO_p^\times$ in $G_p$ when $p \notin S$.  This allows us to choose isomorphisms $K_p \simeq PGL(2,\Z_p)$ when $p \notin S$.  When $M, N \ge 1$ are prime to $S$, we shall use these isomorphisms to define the upper triangular congruence subgroup $K_0(N)$, principal congruence subgroup $K(N)$, and

\bes
K(M,N) = \left\{ k \in K : k \equiv \left( \begin{array}{cc} * & * \\ 0 & * \end{array} \right) (M), k \equiv \left( \begin{array}{cc} * & 0 \\ * & * \end{array} \right) (N) \right\}
\ees
in the natural way.  We choose a maximal compact subgroup $K_\infty \subset G_\infty$.\\

We fix a Haar measure on $G(\A)$ by taking the product of the measures on $G_p$ assigning mass 1 to $K_p$, and any Haar measure on $G_\infty$.  We use this measure to define convolution of functions on $G(\A)$, which we denote by $*$, and if $f \in C^\infty_0(G(\A))$ we use it to define the operator $R(f)$ by which $f$ acts on $L^2(X)$.  If $H$ is a group and $f$ is a function on $H$, we define the function $f^{\vee}$ by $f^{\vee}(h) = \overline{f}(h^{-1})$.  If $f \in C^\infty_0(G(\A))$, the operators $R(f)$ and $R(f^{\vee})$ are adjoints.

\subsection{Newforms}

Let $N \ge 1$ be prime to $S$.  We shall say that $\phi \in L^2(X)$ is a newform of level $K_0(N)$ if $\phi$ lies in an automorphic representation $\pi = \otimes_v \pi_v$ of $G$, $\phi$ is invariant under $K_0(N)$, and we have a factorisation $\phi = \otimes_v \phi_v$ where $\phi_v$ is a local newvector of level $N_v$ for $v \notin S_\infty$.  We shall say that $\phi$ is spherical with spectral parameter $t \in \C$ if $\pi_\infty$ satisfies these conditions, and $\phi$ is invariant under $K_\infty$.  Note that our normalisation of $t$ is such that the tempered principal series corresponds to $t \in \R$.

\subsection{The Harish-Chandra transform}

If $k \in C^\infty_0(G_\infty)$, we define its Harish-Chandra transform

\bes
\widehat{k}(t) = \int_{G_\infty} k(g) \varphi_t(g) dg
\ees
for $t \in \C$, where $\varphi_t$ is the standard spherical function with spectral parameter $t$.  We will use the following standard result on the existence of a $K_\infty$-biinvariant function with concentrated spectral support.

\begin{lemma}
\label{HC}

There is a compact set $B \subset G_\infty$ such that for any $t \in \R \cup [0,i/2]$, there is a $K_\infty$-biinvariant function $k \in C^\infty_0(G_\infty)$ with the following properties:

\begin{enumerate}[(a)]

\item \label{a}
The function $k$ is supported in $B$, and $\| k \|_\infty \ll 1 + |t|$ where the implied constant is uniform in $t$.

\item \label{b}
The Harish-Chandra transform $\widehat{k}$ is non-negative on $\R \cup [0,i/2]$, and satisfies $\widehat{k}(t) \ge 1$.

\end{enumerate}

\end{lemma}

\begin{proof}

When $t \in \R$ and $|t| \ge 1$, this is e.g. Lemma 2.1 of \cite{T2}.  When $|t| \le 1$, one may fix a $K_\infty$-biinvariant real bump function $k_0$ supported near the identity and define $k = k_0 * k_0$.

\end{proof}

Note that condition (\ref{b}) implies that $k = k^{\vee}$.

\subsection{Inner products of matrix coefficients}

Let $H$ be a finite group, and $(\rho, V)$ an irreducible representation of $H$.  We normalise the Haar measure on $H$ to have mass 1.  If $v_i \in V$ for $1 \le i \le 4$, we have

\be
\label{innerprod}
\int_H \langle \rho(h)v_1, v_2 \rangle \overline{ \langle \rho(h) v_3, v_4 \rangle } dh = \frac{ \langle v_1, v_3 \rangle \overline{ \langle v_2, v_4 \rangle } }{ \dim V}.
\ee

\section{Proof of Theorem \ref{level}}
\label{sec3}

\begin{proof}

Choose $g \in G(\A)$ by setting $g_v = 1$ for $v \in S_\infty$, and

\bes
g_v = \left( \begin{array}{cc} N_v & \\ & N_{0,v} \end{array} \right)
\ees
when $v \notin S_\infty$.  We define $\phi' = \pi(g) \phi$, so that $\phi'$ is invariant under $K(N_0, N/N_0)$.  Let $V = \otimes V_p \subset \pi$ be the space generated by $\phi'$ under the action of $K$.

We define $k_f \in C^\infty_0( G(\A_f))$ to be $\overline{\langle \pi(g) \phi', \phi' \rangle}$ for $g \in K$ and 0 otherwise.  Choose a function $k_\infty \in C^\infty_0(G_\infty)$ as in Lemma \ref{HC}, and define $k = k_\infty k_f$.  It may be seen that $k = k^{\vee}$, which implies that $R(k)$ is self-adjoint.  Lemma \ref{dimension} and (\ref{innerprod}) imply that $k_f = \dim V k_f * k_f$, and combining this with Lemma \ref{HC} (\ref{b}) gives that $R(k)$ is non-negative.  Lemmas \ref{HC} and \ref{dimension} and equation (\ref{innerprod}) imply that $\pi(k)\phi' = \lambda \phi'$, where $\lambda > 0$ and

\be
\label{lambda}
\lambda^{-1} \le N_0 \prod_{p | N} (1 + 1/p).
\ee
Extend $\phi'$ to an orthonormal basis $\{ \phi_i \}$ of eigenfunctions for $R(k)$ with eigenvalues $\lambda_i \ge 0$.  The pre-trace formula associated to $k$ is

\bes
\sum_i \lambda_i |\phi_i(x)|^2 = \sum_{\gamma \in G(\Q)} k(x^{-1} \gamma x)
\ees
and dropping all terms from the LHS but $\phi'$ gives

\bes
\lambda |\phi'(x)|^2 \le \sum_{\gamma \in G(\Q)} k(x^{-1} \gamma x).
\ees
The compactness of $X$ and uniformly bounded support of $k$ implies that the number of nonzero terms on the RHS is bounded independently of $x$, and combining (\ref{lambda}) and Lemma \ref{HC} (\ref{b}) completes the proof.

\end{proof}

\begin{lemma}
\label{dimension}

$V$ is an irreducible representation of $K$, and $\dim V \le N_0 \prod_{p | N} (1 + 1/p)$

\end{lemma}

It suffices to prove the analogous statement for the tensor factors $V_p$, and we may assume that $p \notin S$.  If we could write $V_p = V^1 + V^2$, where $V^i$ are nontrivial $K_p$-invariant subspaces, then the projection of $\phi_p'$ to each subspace would be invariant under $K_p(N_0, N/N_0)$.  However, this contradicts the uniqueness of the newvector.

As $V_p$ is irreducible and factors through $K_p / K_p(N_0)$, the Lemma now follows from the results of Silberger \cite[$\mathsection$3.4]{S}, in particular the remarks on page 96-7.

\subsection{A result in the non-compact case}
\label{noncomp}

If we set $G = PGL_2/\Q$, it may be seen that we have the following analogue of Theorem \ref{level}.

\begin{prop}

Let $\Omega \subset G(\Q) \backslash G(\A)$ be compact.  Let $\phi$ be an $L^2$-normalised newform of level $K_0(N)$ on $G$.  Assume that $\phi$ is spherical at infinity with spectral parameter $t$.  Let $N_0 \ge 1$ be the smallest number with $N | N_0^2$.  If $\phi'$ is related to $\phi$ as above, we have $\| \phi'|_\Omega \|_\infty \ll (1 + |t|)^{1/2} N_0^{1/2} \prod_{p | N} (1 + 1/p)^{1/2}$.

\end{prop}

It may be possible to strengthen this to a bound on all of $G(\Q) \backslash G(\A)$ using an analysis of Whittaker functions.

\section{Proof of Theorem \ref{Sogge}}
\label{sec4}

We maintain the notation $k_f$, $\phi'$ and $V$ from $\mathsection$\ref{sec3}.  We note that $\dim V = q$ or $q+1$; the possibility of $\dim V = 1$ is ruled out because any one-dimensional representation of $K_q$ that is trivial on $K_q(q,q)$ must be trivial.  As in Lemma \ref{HC}, choose a $K_\infty$-biinvariant function $k_\infty^0 \in C^\infty_0(G_\infty)$ with the following properties:

\begin{enumerate}[(a)]

\item The function $k_\infty^0$ is real-valued.

\item Its Harish-Chandra transform satisfies $\widehat{k}_\infty^0(t) \ge 1$ for $t \in \R \cup [0,i/2]$ with $|t| \le A$.

\end{enumerate}

It follows that $k_\infty^0 =  (k_\infty^0)^{\vee}$.  We define $k_\infty = k_\infty^0 * k_\infty^0$.  Let $k_0 = k_\infty^0 k_f$, and $k = k_\infty k_f$.  Let $T_0 = R(k_0)$ and $T = R(k)$.  We see that $T_0$ is self-adjoint, and equation (\ref{innerprod}) implies that $T = \dim V T_0^2$.  Let $W \subset K_q$ be the subgroup $\{ 1, w \}$, where

\bes
w = \left( \begin{array}{cc} 0 & 1 \\ -1 & 0 \end{array} \right).
\ees
Let $k_{1,f} \in C^\infty_0(G(\A_f))$ be $k_f$ times the characteristic function of $W K(q,q)$, and let $k_{2,f} = k_f - k_{1,f}$.  Let $k_i = k_\infty k_{i,f}$, and $T_i = R(k_i)$.  The proof of Theorem \ref{Sogge} works by combining the decomposition $T = T_1 + T_2$ with interpolation between the following bounds.

\begin{lemma}

We have

\begin{align*}
\| T_1 f \|_\infty & \ll \| f \|_{1} \\
\| T_2 f \|_\infty & \ll q^{-1/2} \| f \|_{1}
\end{align*}
for any $f \in C^\infty(X)$.

\end{lemma}

\begin{proof}

The integral kernels of $T_i$ are given by

\bes
\sum_{\gamma \in G(\Q)} k_i(x^{-1} \gamma y).
\ees
The result now follows from the compactness of $G(\Q) \backslash G(\A)$, the bound $\| k_1 \|_\infty \ll 1$, and the bound $\| k_2 \|_\infty \ll q^{-1/2}$ which follows from Lemma \ref{matrix} below.

\end{proof}

\begin{lemma}

We have

\begin{align*}
\| T_1 f \|_2 & \ll q^{-2} \| f \|_2 \\
\| T_2 f \|_2 & \ll q^{-1} \| f \|_2
\end{align*}
for any $f \in C^\infty(X)$.

\end{lemma}

\begin{proof}

The choice of $k_\infty$ and the identity $k_f = \dim V k_f * k_f$ imply that the $L^2 \rightarrow L^2$ norm of $T$ is $\ll (\dim V)^{-1} \le q^{-1}$.  In the same way, the identity

\bes
k_{1,f} = [ K : W K(q,q) ] k_{1,f} * k_{1,f} = (q(q+1)/2) k_{1,f} * k_{1,f}
\ees
implies that the $L^2 \rightarrow L^2$ norm of $T_1$ is $\ll q^{-2}$.  The bound for $T_2$ follows from the triangle inequality.

\end{proof}

Interpolating between these bounds and applying Minkowski's inequality $\| Tf \|_p \le \| T_1f \|_p + \| T_2f \|_p$ gives $\| Tf \|_p \ll q^{2\delta(p) - 1} \| f \|_{p'}$.  We now apply the usual adjoint-square argument: we have

\begin{align*}
\langle \dim V T_0^2 f, f \rangle & \ll q^{2\delta(p) - 1} \| f \|_{p'}^2 \\
\langle T_0 f, T_0 f \rangle & \ll q^{2\delta(p) - 2} \| f \|_{p'}^2 \\
\| T_0 f \|_2 & \ll q^{\delta(p) - 1} \| f \|_{p'}.
\end{align*}
Taking adjoints gives $\| T_0 f \|_p \ll q^{\delta(p) - 1} \| f \|_2$.  Applying this with $f = \phi'$ and estimating the eigenvalue of $T_0$ on $\phi'$ as in $\mathsection$\ref{sec3} completes the proof.

\begin{lemma}
\label{matrix}

Let $\pi_q$ be isomorphic to an irreducible principal series representation $\cI(\chi, \chi^{-1})$, for some character $\chi$ of $\Q_q$ with conductor $q$.  When $g \in K_q$, the matrix coefficient $\langle \pi_q(g) \phi'_q, \phi'_q \rangle$ satisfies

\begin{align}
\label{matrix1}
\langle \pi_q(g) \phi'_q, \phi'_q \rangle & = 1, \qquad g \in K_q(q,q) \\
\label{matrix2}
\langle \pi_q(g) \phi'_q, \phi'_q \rangle & = \chi(-1), \qquad g \in wK_q(q,q) \\
\label{matrix3}
\langle \pi_q(g) \phi'_q, \phi'_q \rangle & \ll q^{-1/2}, \qquad g \notin W K_q(q,q)
\end{align}
where the implied constant is absolute.

\end{lemma}

\begin{proof}

We may reduce the problem to one for the group $PGL(2,\F_q)$, where $\F_q$ denotes the finite field of $q$ elements.  We let $T$ and $B$ be the diagonal and Borel subgroups of $PGL(2,\F_q)$.  We now think of $\chi$ as a non-trivial character of $\F_q^\times$, and let $(\rho, H)$ denote the corresponding principal series representation of $PGL(2,\F_q)$.  We realise $H$ as the space of functions $f : PGL(2,\F_q) \rightarrow \C$ satisfying

\bes
f \left( \left( \begin{array}{cc} a & b \\ 0 & d \end{array} \right) g \right) = \chi(a/d) f(g)
\ees
with the norm

\bes
\| f \| = \sum_{ g \in B \backslash PGL(2,\F_q) } |f(g)|^2.
\ees

It may be seen that there is a unique vector $v \in H$ that is invariant under $T$, up to scaling, and we may choose it to be represented by the function

\bes
\left( \begin{array}{cc} a & b \\ c & d \end{array} \right) \mapsto \Big\{ \begin{array}{ll} \chi( \det / cd) / \sqrt{q-1}, & cd \neq 0, \\ 0, & cd = 0. \end{array}
\ees
It follows that $\| v \| = 1$, and so $\langle \pi_q(g) \phi'_q, \phi'_q \rangle = \langle \rho(g) v, v \rangle$.  Equation (\ref{matrix1}) is immediate, and (\ref{matrix2}) follows from $\rho(w) v = \chi(-1)v$.  We now assume that

\bes
g = \left( \begin{array}{cc} a & b \\ c & d \end{array} \right) \notin WT.
\ees
An elementary calculation gives

\bes
\langle \rho(g) v, v \rangle = (q-1)^{-1} \chi^{-1}(\det(g)) \sum_n \chi( (c +an)(d + bn) ) \chi^{-1}(n).
\ees
We bound this sum by rewriting it as

\bes
\sum_n \chi( (c +an)(d + bn) n^{q-2} )
\ees
and applying \cite[Ch. 2, Theorem 2.4]{Sc} (see also \cite[Theorem 11.23]{IK}).  We must first check that $(c +an)(d + bn) n^{q-2}$ is not a proper power.  The assumption $g \notin WT$ implies that one or both of $a + cn$ and $b + dn$ have a root distinct from 0.  If they both have the same root distinct from 0, this contradicts the invertability of $g$.  Therefore $(c +an)(d + bn) n^{q-2}$ must have at least one root of multiplicity 1, so it cannot be a power.  As $(c +an)(d + bn) n^{q-2}$ has at most 3 distinct roots, \cite{Sc} or \cite{IK} therefore give

\bes
| \sum_n \chi( (c +an)(d + bn) n^{q-2} ) | \le 2 \sqrt{q},
\ees
which completes the proof of (\ref{matrix3}).

\end{proof}

\end{document}